\documentclass{article}
\usepackage[T1]{fontenc}
\usepackage[utf8]{inputenc}
\usepackage[english]{babel}

\usepackage{authblk}
\usepackage{booktabs}

\usepackage[utf8]{inputenc}
\usepackage[T1]{fontenc}
\usepackage[english]{babel}
\usepackage{textcomp}
\usepackage{amsmath,amssymb,amsthm}
\usepackage{lmodern}
\usepackage[a4paper]{geometry}
\usepackage{xcolor, pict2e}
\usepackage{microtype}
\usepackage{caption}
\usepackage{listings}
\usepackage{multicol}
\usepackage{moreverb}
\usepackage{hyperref}
\hypersetup{pdfstartview=XYZ}
\usepackage{wrapfig}
\usepackage[sans]{dsfont}

\usepackage{tikz, pgfplots}
\usetikzlibrary{positioning}
\usepackage{amsmath,amssymb}
\usetikzlibrary{decorations.pathreplacing}

\usepackage{ytableau}

\usepackage{mathdots}

\usepackage{hyperref}
\hypersetup{
    colorlinks=true,
    linkcolor=blue,
    citecolor=magenta,
    urlcolor=blue,
    pdfborder={0 0 0}
}

\definecolor{fond}{rgb}{0.05,0.05,0.25}

\DeclareMathOperator{\Err}{Err}

\DeclareMathSymbol{\eset}{\mathalpha}{AMSb}{"3F}

\usepackage{todonotes}

\numberwithin{equation}{section}

\theoremstyle{definition}
\newtheorem{thm}{Theorem}[section]

\newtheorem{lemma}[thm]{Lemma}
\newtheorem{corollary}[thm]{Corollary}
\newtheorem{prop}[thm]{Proposition}

\newtheorem{remark}[thm]{Remark}

\newcommand{\bbE}{{\ensuremath{\mathbb E}} }
\newcommand{\E}{{\ensuremath{\mathbb E}} } 
\renewcommand{\P}{{\ensuremath{\mathbb P}} } 

\newcommand{\bbP}{{\ensuremath{\mathbb P}} }

\newcommand{\bbR}{{\ensuremath{\mathbb R}} }

\newcommand{\bbZ}{{\ensuremath{\mathbb Z}} }

\newcommand{\ag}{\left\{ } 

\newcommand{\ad}{\right\} }

\newcommand{\cg}{\left[}
\newcommand{\cd}{\right]}
\newcommand{\pg}{\left(} 
\newcommand{\pd}{\right)}

\DeclareMathOperator{\dist}{dist}
\DeclareMathOperator{\tmed}{\textit{t}_{med}}
\DeclareMathOperator{\Tr}{Tr}

\DeclareMathOperator{\trel}{\textit{t}_{rel}}

\DeclareMathOperator{\Diam}{\text{Diam}}

\DeclareMathOperator{\Irr}{\text{Irr}}
\DeclareMathOperator{\MaxIrr}{\text{MaxIrr}}

\usepackage[normalem]{ulem}

\numberwithin{equation}{section}

\usepackage{pgfplots}
\pgfplotsset{compat=1.18}

\begin{document}

\title{Stationary hitting times on vertex-transitive graphs}
\date{}

\author[1]{Nathanaël Berestycki}
\author[2]{Jonathan Hermon}
\author[3]{Lucas Teyssier}
\affil[1]{Universität Wien, \texttt{nathanael.berestycki@univie.ac.at}}
\affil[2]{University of British Columbia, \texttt{jhermon@math.ubc.ca}}
\affil[3]{Université de Lorraine, \texttt{lucas.teyssier@univ-lorraine.fr}}

\maketitle








\begin{abstract}
We prove a refined version of the Aldous and Brown's exponential approximation of stationary hitting times. These are valid for all reversible Markov chains. 
We then specialise our estimates for vertex-transitive graphs, where we obtain improved bounds which depend on the growth of the graphs. The most delicate cases are when the diameter is comparable to that of low-dimensional tori. In particular, in “dimensions” less than four (up to logarithmic factors) our error terms are the square of those of Aldous and Brown. These improved bounds play a crucial role in the companion work \cite{BerestyckiHermonTeyssier2022CTU} characterising the fluctuations of the cover time on vertex-transitive graphs.
\end{abstract}

\maketitle



\section{Introduction}\label{S:intro}
\subsection{Main results}

Let $X = (X_t)_{t\geq 0}$ be an irreducible reversible (rate 1) continuous-time Markov chain on a finite state space $S$. Denote the hitting time of a vertex $x\in S$ by $T_x := \min \ag t\geq 0 \mid X_t = x \ad$. The \textbf{relaxation time} of the chain is defined by $\trel := 1/(1-\theta_2)$, where $1 = \theta_1>\theta_2> \ldots$ are the eigenvalues of the transition matrix $P$ of the chain.
Let $\eset \ne B\subsetneq S$ such that the restriction of the chain to $B$ is irreducible, i.e.\ such that $\mathbb{P}_{b_1}[T_{b_2}<T_A]>0$ for all $b_1, b_2\in B$, where $T_A := \max_{x\in A} T_x$ is the hitting time of the set $A:= B^c$. We will say that such a set $B$ is $X$-\textbf{irreducible}.
It is known that there exists a unique probability measure $\alpha_B$ supported on $B$, called the \textbf{quasi-stationary distribution} on $B$, 
such that $T_A$ is an exponential variable for the chain conditioned on starting at $\alpha_B$: $T_A$ satisfies, for $t\geq 0$,
\begin{equation}
\label{e:QS tail}
    \bbP_{\alpha_B}(T_A>t) = \exp\pg - \frac{t}{\bbE_{\alpha_B}\cg T_A \cd}\pd \, .
\end{equation}
Quantifying how close quasi-stationary distributions are from the stationary distribution $\pi$ is fundamental for approximating hitting times $T_A$ for the walk started at $\pi$ by exponential variables. Such exponential approximations go back to the seminal work of Aldous and Brown \cite{AldousBrown1992}, who obtained quantitative error bounds which we now recall.

\begin{thm}[{\cite[Theorem 3, Equation (1), and Corollary 4]{AldousBrown1992}}]
\label{T:ABintro}
Let $(X_t)_{t\geq 0}$ be an irreducible reversible continuous-time Markov chain on a finite state space $S$, and denote its stationary distribution by $\pi$. Let $\eset \ne B\subsetneq S$ be such that the restriction of the chain to $B$ is irreducible, and set $A:= B^c$. We have that
\begin{equation}\label{e:errortermABintroproba}
     \pi(A)\leq 1 - \frac{\bbP_\pi(T_A>t)}{\bbP_{\alpha_B}(T_A>t)} \leq \frac{\trel}{\bbE_{\alpha_B} \cg T_A \cd}
\end{equation}
for all $t\geq 0$, and
\begin{equation}
\label{e:errortermABintroexpectations}
    \pi(A) \leq 1-\frac{\bbE_\pi \cg T_A \cd}{\bbE_{\alpha_B}\cg T_A \cd} \leq \frac{\trel}{\bbE_{\alpha_B} \cg T_A \cd } \, .
\end{equation}
\end{thm}

Our main results are improvements of the “error term” in the right hand side of \eqref{e:errortermABintroproba} and \eqref{e:errortermABintroexpectations}. Theorem \ref{thm: general bound bound with relaxation time intro} proves an improved bound for reversible Markov chains. Theorems \ref{thm: bounds QS for vertex transitive graphs intro with integrals and volumes} and \ref{thm:main improvement AB transitive} prove concrete bounds for transitive graphs depending on their volume growth. At the same time we remove the condition that the restriction of $P$ to $B$ is irreducible from all statements.

\medskip

The results below were motivated by the study of cover times for vertex-transitive graphs in the companion paper \cite{BerestyckiHermonTeyssier2022CTU}. There, Aldous--Brown types of approximations are useful to relate the hitting probabilities of small sets to their \textit{capacities}. The improvement that we prove in Theorem \ref{thm:main improvement AB transitive} is necessary to obtain the characterisations in \cite[Theorems 1.1 and 1.2]{BerestyckiHermonTeyssier2022CTU}.

\medskip

There has been much interest in recent years in understanding the quasi-stationary distribution, and the rate of convergence to it for the chain conditioned on not hitting the corresponding set, see \cite{DiaconisMiclo2009,DiaconisMiclo2015,DiaconisHouston-EdwardsKelseySaloff-Coste2021}. The results below should be useful in that context too.

\medskip

Our first improvement is the following.
\begin{thm}[A refinement of the AB Theorem]
\label{thm: general bound bound with relaxation time intro}
Let $(X_t)_{t\geq 0}$ be an irreducible reversible continuous-time Markov chain on a finite state space $S$, and denote its  stationary distribution by $\pi$. Let $\eset \ne A\subsetneq S$. Let $M$ be an $X$-irreducible component of $A^c$ for which $\bbE_{\alpha_M}[T_A]$ is maximal. Then
\begin{equation}
\label{e:ABrefinement1trelintro}
    \pi(A) \le 1 - \frac{\bbP_\pi(T_A>t)}{\bbP_{\alpha_M}(T_A>t)}\le \pi(A) + 2 \sum_{x\in S} \pi(x) \left(   \mathbb{P}_x[T_A \le 2\trel ]  \right)^2 \, ,
\end{equation}
for every $t\geq 0$, and
\begin{equation}
\label{e:ABrefinement4trelintro}
\pi(A)\le 1-\frac{\bbE_\pi \cg T_A \cd}{\bbE_{\alpha_M}\cg T_A \cd}   \le \pi(A) + 2 \sum_{x\in S} \pi(x) \left(   \mathbb{P}_x[T_A \le 2\trel ]  \right)^2 \, .
\end{equation}
\end{thm}

\begin{remark}\label{rem: conditions unique maximal eigenvalue}
    Let $X=(X_t)_{t\geq 0}$, $A$ and $M$ as in Theorem \ref{thm: general bound bound with relaxation time intro}.
    Denote the set of $X$-irreducible components of a subset $W\subsetneq S$ by $\Irr_X(W)$.
    We will see as a consequence of the Perron--Frobenius theorem and the interlacing eigenvalue theorem that we can have $\bbE_{\alpha_I}[T_A]>\trel$ for at most one $I\in \Irr_X(A^c)$. If 
    \begin{equation}\label{eq:condition max E alpha I trel}
        \max_{I\in \Irr_X(A^c)} \bbE_{\alpha_I}[T_A]>\trel \, ,
    \end{equation} then $M$ is \textit{the unique} such maximiser.

In most interesting cases (for example for finite sets in large tori), the condition \eqref{eq:condition max E alpha I trel} is automatically satisfied, and when it fails to hold the upper bounds in Theorems \ref{T:ABintro} and \ref{thm: general bound bound with relaxation time intro} are trivial.
\end{remark}

    For small sets $A$, applying the trivial bound $\mathbb{P}_x[T_A \le 2 \trel]^2 \le \mathbb{P}_x[T_A \le 2 \trel]$ to the right hand side of \eqref{e:ABrefinement1trelintro} and \eqref{e:ABrefinement4trelintro} recovers Theorem \ref{T:ABintro} up to a universal constant. However in general this bound is very wasteful, and it is often possible to estimate directly the error term in Theorem \ref{thm: general bound bound with relaxation time intro} (i.e., the sum in the right hand side of \eqref{e:ABrefinement1trelintro} and \eqref{e:ABrefinement4trelintro}).    
For vertex-transitive graphs, it can be bounded in terms of the growth of the graph as follows.
\begin{thm}\label{thm: bounds QS for vertex transitive graphs intro with integrals and volumes}
    There exists a universal constant $C_0$ such that the following holds. Let $\Gamma = (V,E)$ be a finite connected vertex-transitive graph, with degree $\deg(\Gamma)\geq 1$. Let $\eset \ne A \subsetneq V$.
    For the simple random walk on $\Gamma$, we have  
   \begin{equation}
   \begin{split}
       2\sum_{x\in S} \pi(x) \left(   \mathbb{P}_x[T_A \le 2\trel ]  \right)^2 & \leq C_0|A|^2 \pg \frac{\trel}{\bbE_\pi[T_o]} \pd^2 \pg 1 + \frac{\deg(\Gamma)^2 n}{\trel^2}\int_{0}^{\sqrt{\trel/\deg(\Gamma)}}\frac{r^{3} \mathrm{d}r}{V(r)}  \pd \\
       & =: \Err(\Gamma, A) \, ,
   \end{split}
\end{equation}
where $V(r)$ denotes the volume of a ball of radius $\lfloor r \rfloor$.
\end{thm}

This type of improvement is crucial to handle borderline cases, and is in fact significant for tori $(\bbZ/m\bbZ)^d$ for any $d\geq 2$. 
Indeed for tori (if $d$ is fixed and as $m\to \infty$) we have $\trel \asymp m^2$, and $V(r)\asymp r^d$. In particular for tori of dimensions 2 and 3, for sets $A$ of bounded size, we have
\begin{equation}
    \Err(\Gamma, A) \asymp \pg \frac{\trel}{\bbE_\pi T_0}\pd^2  \asymp \pg \frac{\trel}{\bbE_{\alpha_{M}} T_A}\pd^2,
\end{equation} 
which is the square of the error term of Aldous and Brown!
The following table illustrates the improvements of the error terms for tori.
\begin{table}[h]
\centering
\caption{Asymptotic parameters for tori $(\mathbb{Z}/m\mathbb{Z})^d$ as $m \to \infty$ for sets $A$ of size $k\geq 1$. The column 2 is $\Theta_d(\cdot)$, and the columns 3 and 4 are $\Theta_{d,k}(\cdot)$. In all cases we have $\trel =\Theta_d(m^2)$. We denote $n = |V| = m^d$. }
\begin{tabular}{cccccc}
\toprule
Dimension $d$  & Hitting time $\bbE_\pi T_0$ & Error term $\trel/\bbE_\pi T_A$ & $\Err(\Gamma, A)$ \\
\midrule
1            & $m^2$         & $1$           & $1$  \\
2             & $m^2\log m$   & $1/\log m$    & $(1/\log m)^2$  \\
3             & $n = m^3$         & $n^{-1+2/d} = 1/m$         & $(1/m)^2$  \\
4              & $n = m^4$         & $n^{-1+2/d} = 1/m^2$       & $\frac{\log n}{n} \asymp (\log m)/m^4$  \\
$\geq 5$       & $n= m^d$         & $n^{-1+2/d} = 1/m^{d-2}$   & $1/n = 1/m^d$   \\
\bottomrule
\end{tabular}
\end{table}

Finally, using the recent structural results of Tessera and Tointon \cite{TesseraTointonfinitary, TesseraTointonIsop}, we are able to further bound our error terms $\Err(\Gamma, A)$ for vertex-transitive graphs, purely in terms of the diameter of the graph. Let us set up some notation. Let $\Gamma = (V,E)$ be a finite connected vertex-transitive graph. Denote its number of vertices by $n = |V|$, its graph distance by $\dist_{\Gamma}$, and its diameter by $D = \Diam(\Gamma) := \max_{x,y\in V} \dist_{\Gamma}(x,y)$. Denote its degree by $\deg(\Gamma)$. We can write in a unique way $n = D^q R$, where $q$ is an integer and $1\leq R<D$. Consider the simple random walk on $\Gamma$. Denote its stationary distribution by $\pi$ and let $o\in V$ be any vertex. We set
\begin{equation}\label{e:def of beta Gamma}
    \beta(\Gamma) := \begin{cases}
\frac{D^4}{(\mathbb{E}_{\pi}[T_o])^{2}} & \text{ if } q \in \ag 1,2\ad \\
\frac{D^4}{(\mathbb{E}_{\pi}[T_o])^{2}}\pg 1 + \frac{R\log R}{D}\pd & \text{ if } q=3 \\
\frac{D^4}{(\mathbb{E}_{\pi}[T_o])^{2}}\pg R + \log (\frac{D}{R})\pd & \text{ if } q=4 \\
1/|V| & \text{ if } q \ge 5  \\
\end{cases}.
\end{equation}
We prove the following.
\begin{thm}[A specialised quasi-stationary approximation for transitive graphs]
\label{thm:main improvement AB transitive}
     There exists a universal constant $C_1>0$ such that the following holds. Let $d\geq 2$ be an integer. Let $\Gamma = (V,E)$ be a finite connected vertex-transitive graph of degree $d$. Let $\eset \ne A\subsetneq V$. Let $M$ be an $X$-irreducible component of $A^c$ for which $\bbE_{\alpha_M}[T_A]$ is maximal.\footnote{By Remark \ref{rem: conditions unique maximal eigenvalue}, $M$ is automatically guaranteed to be the unique maximiser if $\bbE_{\alpha_{M}}[T_A]>\trel$.}
   For the simple random walk on $\Gamma$, we have
\begin{equation}
\label{e:improvementABtransitivewithbetaproba}
   0 \leq \pi(M) - \frac{\bbP_\pi(T_A>t)}{\bbP_{\alpha_{M}}(T_A>t)} \leq C_1|A|^2 d^2\beta(\Gamma)
\end{equation}
for all $t\geq 0$, and 
\begin{equation}\label{e:improvementABtransitivewithbetaexpectations}
    0\leq \pi(M) - \frac{\bbE_\pi \cg T_A \cd}{\bbE_{\alpha_{M}} \cg T_A \cd} \leq C_1|A|^2d^2\beta(\Gamma) \, .
\end{equation}
\end{thm}
\begin{remark}[Dependence on the degree]
    \label{rem:dependence on |A| and d}
 We believe the optimal dependence on $d$ in Theorem \ref{thm:main improvement AB transitive} can be improved to be linear.
 For edge-transitive graphs we believe that the dependence on $d$ can be entirely removed.
\end{remark}

\begin{remark}
    Our results also hold in discrete time, up to a few minor changes which are described in Remark \ref{rem: discrete and continuous}.
\end{remark}

Many arguments from Section \ref{s:preparation} are borrowed from \cite{AldousBrown1992}. In order to prove the refined bounds from Theorem \ref{thm: general bound bound with relaxation time intro}, there are two necessary new steps. The first one is to find an explicit formula for the first coefficient of the decomposition of the indicator function $\mathrm{1}_B$ in an appropriate orthonormal basis; this is presented in Lemma \ref{lem: new identity c 1 square}. The second one is the introduction of the auxiliary time $\tmed$ in \eqref{e: definition of t med} and its study. 

\section{Quasi-stationary distributions and hitting times: a simple approach}
\label{s:preparation}
In this section $X = (X_t)_{t\geq 0}$ is an irreducible reversible (rate 1) continuous-time Markov chain on a finite state space $S$ with transition matrix $P$ and stationary measure $\pi$ (thus the jump rate from every $x \in S$ is equal to 1, and at such a jump time, the chain jumps from $x$ to $y$ with probability $P(x,y)$).  
For $f,g:S \to \mathbb{R}$, the  expectation of $f$ is defined by $\mathbb{E}_{\pi}[f]:=\sum_{x \in S}\pi(x)f(x)$, the inner-product of $f$ and $g$ by $\langle f,g \rangle_{\pi}:=\mathbb{E}_{\pi}[fg]$, and the associated $L^2$ norm by $\|f\|_2:=\sqrt{\langle f,f \rangle_{\pi}}$.

This section presents an efficient self-contained approach to quasi-stationary distributions. Many arguments are -- unsurprisingly -- borrowed from \cite{AldousBrown1992}, as well as \cite[Section 3.6.5]{AldousFill} and \cite[Section 6.9]{Keilson1979book}. What makes our approach simpler is that we make a better use of the $L^2$ structure (with respect to the inner product $\langle \cdot, \cdot \rangle_\pi$) of the spaces  
\begin{equation}
    C_0(B):=\{f \in \mathbb{R}^S:f(x)=0 \text{ for all }a \in B^c \} \, ,
\end{equation}
for subsets $\eset \ne B \subsetneq S$.

\subsection{Killed chains and quasi-stationary distributions}\label{s:Killed chains and quasi-stationary distributions}
Let $\eset \ne B\subsetneq S$. Set, for $x,y\in S$,
\begin{equation}
P_B(x,y):=P(x,y)\mathbf{1}\{x,y \in B \} \, .
\end{equation}
 This defines a strictly sub-stochastic matrix $P_B$ on $S\times S$. In particular, if $x,y\in S$ are such that $x\in A$ or $y\in A$, we have $P_B(x,y) = 0$. The matrix $P_B$ is the transition matrix of the chain \textit{killed} (or \textit{frozen}) upon hitting $A=B^c$.

\begin{lemma}
    Let $\eset \ne B\subsetneq S$. $P_B$ is self-adjoint with respect to the inner-product $\langle \cdot , \cdot \rangle_{\pi}$ on $C_0(B)$.
\end{lemma}

\begin{proof}
    Let $x,y\in S$. By definition of reversibility, we have $\pi(x)P(x,y) = \pi(y)P(y,x)$. Therefore 
\begin{equation}\label{e: P_B form of balanced equations with respect to pi}
    \pi(x)P_B(x,y) = \pi(x)P(x,y)\mathbf{1}\{x,y \in B \} = \pi(y)P(y,x)\mathbf{1}\{x,y \in B \} = \pi(y)P_B(y,x) \, .
\end{equation}    
Now, let $f,g \in  C_0(B)$. Reordering terms and using \eqref{e: P_B form of balanced equations with respect to pi}, we obtain 
    \begin{equation}
    \begin{split}
        \langle f , P_B g \rangle_{\pi} (P_B g)(x) 
        & = \sum_{x\in B} \pi(x)f(x) \sum_{y\in B}P_B(x,y)g(y) \\ 
        & = \sum_{x,y\in B} \cg \pi(y)P_B(y,x)\cd g(y) f(x) \\
        & = \sum_{y\in B} \pi(y) g(y) \sum_{x\in B}  P_B(y,x) f(x) = \langle  g, P_B f \rangle_{\pi} \, ,
    \end{split} 
    \end{equation}
i.e.\ that $P_B$ is self-adjoint.
\end{proof}

Consider the restriction $\widetilde{P}_B$ of $P_B$ to $B\times B$ (obtained by deleting the rows and columns corresponding to $A$). Since $P$ is an irreducible stochastic matrix, $\widetilde{P}_B$ is strictly sub-stochastic. In particular, the eigenvalues of  $\widetilde{P}_B$ are strictly less than 1 in modulus. 

\medskip

\textbf{Assume moreover that $B$ is $X$-irreducible.}
Then the matrix $\widetilde{P}_B$ is primitive.
By the Perron--Frobenius theorem, the spectral radius $\rho$ of $\widetilde{P}_B$ is an eigenvalue of $P_B$ whose associated eigenspace is a line that contains a unique probability measure $\alpha_B$ on $B$.
The probability measure $\alpha_B$ (which naturally extends to $S$ by setting $\alpha_B(x) = 0$ for $x\in B^c$) is called the \textbf{quasi-stationary distribution} for the set $B$.
To summarise this in an equation, we have
\begin{equation}\label{eq: alpha left eigenvector}
    \alpha_B P_B = \rho \alpha_B \, .
\end{equation}
Moreover, since $P_B$ is $\langle \cdot , \cdot \rangle_{\pi}$-self-adjoint on $C_0(B)$, there exist $1>\rho = \gamma_1 = \gamma_1(B) > \gamma_2 = \gamma_2(B) \geq \ldots \geq \gamma_m = \gamma_m(B)$, where  $m:= |B|$; and an $\langle \cdot, \cdot \rangle_{\pi}$-orthonormal basis $(f_1,\ldots,f_m)$ basis of $C_0(B)$, such that $P_B f_i=\gamma_i f_i$ for all $1\leq i\leq m$. 
(Note that $f_1, \ldots, f_m$ are right eigenvectors of $P_B$.) Denote also $\lambda_i = \lambda_i(B):= 1-\gamma_i(B)$ for $1\leq i\leq m$.

Finally, the Radon--Nikodym derivative $\alpha_B/\pi:S \to \mathbb{R}$ of $\alpha$ with respect to $\pi$ is defined by
\begin{equation}
    (\alpha_B/\pi)(a):= \begin{cases}
        \alpha_B(a)/\pi(a) & \text{ if } a\in B \, , \\
        0 & \text{ if }  a \in B^c \, .
    \end{cases}
\end{equation}
Then for any $x\in S$, by \eqref{e: P_B form of balanced equations with respect to pi} and since $\alpha_B P_B = \rho \alpha_B$ by \eqref{eq: alpha left eigenvector}, we have
    \begin{equation}
    \begin{split}
          \pg P_B\frac{\alpha_B}{\pi}\pd (x) = \sum_{y\in B}P_B(x,y) \frac{\alpha_B(y)}{\pi(y)} = \sum_{y\in B}P_B(y,x) \frac{\alpha_B(y)}{\pi(x)} & = \frac{1}{\pi(x)} \sum_{y\in B} \alpha_B(y)P_B(y,x) \\
          & = \frac{\rho \alpha_B(x)}{\pi(x)} \, ,
    \end{split}
    \end{equation}
    that is, $ P_B(\alpha_B/\pi) = \rho (\alpha_B/\pi)$.
In other words the vector $\alpha_B/\pi$ is a right eigenvector of $P_B$ with eigenvalue $\rho$. Up to replacing $f_1$ by $-f_1$ to ensure that $f_1$ has at least one positive coefficient, we therefore have
   \begin{equation}\label{e: f_1 is alpha rescaled}
    f_1 = \frac{\alpha_B/\pi}{\|\alpha_B/\pi\|_2} \, .
\end{equation}

Define also a matrix $I_B$ by $I_B(x,y) = \mathbf{1}\{x=y \in B \}$ for $x,y\in S$, and write $\lambda = \lambda_1= 1-\rho$. In words, $\lambda$ is the smallest eigenvalue of $I_{B}-P_B$, and $Q_B:=P_B-I_B$ is the Markov generator corresponding to the rate 1 continuous-time version of this killed chain.

\medskip

We now recall (and reprove) that the quasi-stationary distribution $\alpha_B$ corresponding to the set $B$ satisfies that the first hitting time of $A=B^c$ under $\P_{\alpha_B}$ has (in continuous-time) an exponential distribution with parameter $\lambda = \lambda_1 (B)$, and that furthermore,
\begin{equation}\label{e: identity lambda gamma 1}
\lambda_1(B) = 1/\bbE_{\alpha_B}\cg T_A \cd.  
\end{equation}

\begin{lemma}\label{lem: QS reminders}
Let $\eset \ne B \subsetneq S$ be $X$-irreducible. Write $\alpha$ for $\alpha_B$. 
We have $\lambda = 1/\bbE_\alpha\cg T_A \cd$, and for any $t\geq 0$,
    \begin{equation}
        \bbP_\alpha(T_A>t) = e^{- t/\bbE_\alpha\cg T_A \cd} =e^{-\lambda t} \, .
    \end{equation}
\end{lemma}
\begin{proof}
Let $t\geq 0$. Since $\alpha Q_B = -\lambda \alpha$, we have $\alpha e^{tQ_B} =e^{-\lambda t} \alpha$.
Therefore, denoting by $X_t^A$ the position at time $t$ of the chain killed (or frozen) upon hitting $A$, 
\begin{equation}
    \mathbb{P}_{\alpha}[T_{A}>t]= \bbP_\alpha(X^{A}_t \in B) = (\alpha e^{tQ_B})(B)=e^{-\lambda t}\alpha(B) = e^{-\lambda t} \, ,
\end{equation}
i.e.\ starting from $\alpha$ the law of $T_A$ is exponential with mean $1/\lambda$. Taking expectations we finally obtain that $1/\lambda=\mathbb{E}_{\alpha}[T_A]$. 
\end{proof}
\begin{remark}
With the same notation as in Lemma \ref{lem: QS reminders}, for all $x \in B$ and all $t \ge 0$, we have
    \begin{equation}
         (\alpha e^{tQ_B})(x) = \bbP_\alpha(X_t^A = x) =  \mathbb{P}_{\alpha}[X_{t}=x,T_{A}>t]  = \mathbb{P}_{\alpha}[X_{t}=x \mid T_{A}>t] \mathbb{P}_{\alpha}[T_{A}>t] \, ,
    \end{equation}
    i.e.\ since $\alpha e^{tQ_B}=e^{-\lambda t }\alpha$ and $\mathbb{P}_{\alpha}[T_{A}>t] = e^{-\lambda t}$,
\begin{equation}
    \mathbb{P}_{\alpha}[X_{t}=x \mid T_{A}>t] =\alpha(x) \, .
\end{equation}
This last equation justifies the name ``quasi-stationary distribution''.
\end{remark}

\textbf{If $B$ is not $X$-irreducible}, we can partition $B$ into $X$-irreducible components $B_1, \ldots, B_p$, and apply the Perron--Frobenius to each component. Then concatenating the orthonormal bases of the $C_0(B_i)$, we obtain an orthonormal basis $(f_1, \ldots, f_m)$ of $C_0(B)$, where again $m=|B|$, with associated eigenvalues $1>\rho = \gamma_1(B)\geq \gamma_2(B)\geq \ldots \geq \gamma_m(B)$. The main difference is that the matrix $\widetilde{P}_B$ is not primitive, and that we may have $\gamma_1(B) = \gamma_2(B)$.
In this case the maximal eigenvalues must come from different irreducible components of $B$ (because by the Perron--Frobenius theorem there is a unique maximal eigenvalue on each component, as described above), and are the eigenvalues $1-1/\bbE_{\alpha_{B_i}} T_A$ of the quasi-stationary distributions of these components. Denote by 
\begin{equation}
    \MaxIrr_X(B) := \ag I \in \Irr_X(B) \mid \bbE_{\alpha_I} T_A  = 1/(1-\gamma_1(B))\ad,
\end{equation}
the set of $X$-irreducible components of $B$ with maximal quasi-stationary hitting time of $B^c$. 
Then rephrasing what is written above, the multiplicity of $\gamma_1(B)$ in $P_B$ is the cardinality of $\MaxIrr_X(B)$.

\begin{remark}\label{rem: unique maximal eigenvalue up to adding laziness locally}
Assume that $|\MaxIrr_X(B)|\geq 2$ and let $M\in \MaxIrr_X(B)$. We may consider for $\varepsilon\in (0,1)$ a modified version $X^{(\varepsilon, M)}$ of the chain $X$ which has an additional $\varepsilon$ little bit of laziness in $M$; that is, the chain has a transition matrix $P^{(\varepsilon, M)}$ satisfying
\begin{equation}
 P^{(\varepsilon, M)}(x,y) =
    \begin{cases}
        \varepsilon + (1-\varepsilon)P(x,y) & \text{ if } x=y\in M,\\
        (1-\varepsilon)P(x,y) &  \text{ if } x \in M \text{ and } y \in S \backslash \ag x \ad,\\
        P(x,y) & \text{ if } x\notin M \text{ and } y \in S.
    \end{cases}
\end{equation}
Then $X$ and $X^{(\varepsilon, M)}$ have the same behaviour, except that $X^{(\varepsilon, M)}$ is slower in $M$. (Indeed adding some laziness to the transition matrix has the same impact as changing the jump rate from $1$ to $(1-\varepsilon)$ when the walk is in $M$.)

Since the hitting time of $A$ from $\alpha_M$ is multiplied by $1/(1-\varepsilon)$ for the walk $X^{(\varepsilon, M)}$ and the hitting time of $A$ from $\alpha_I$ for each $I\in \Irr_X(B)\backslash\ag M \ad$ is unchanged, for $X^{(\varepsilon, M)}$ the value of $\gamma_1(B)$ is a little higher, and hence we have $\gamma_1(B) > \gamma_2(B)$ for the chain $X^{(\varepsilon, M)}$. But as $\varepsilon\to 0$, the probabilities of hitting a set before some time and the expected hitting times for $X^{(\varepsilon, M)}$ tend to those for $X$. Therefore, up to studying the chains $X^{(\varepsilon, M)}$ and then letting $\varepsilon \to 0$, 
\textbf{we may always assume that}
\begin{equation}
    \gamma_1(B) > \gamma_2(B),
\end{equation}
\textbf{that is, that there is a unique set $M\in \MaxIrr_X(B)$, which we denote by $M$ by default in what follows.}
(This is purely for convenience. All proofs work if $|\MaxIrr_X(B)|\geq 2$ and we consider a set $M \in \MaxIrr_X(B)$, but this would require each time specifying that up to reordering, the eigenvector $f_1$ is that corresponding to the chosen $M$.)
\end{remark}

\subsection{A theorem of Aldous and Brown, remastered}

An important quantity is the \textit{quasi-stationary default of stationarity} for an $X$-irreducible subset $M$, defined by
\begin{equation}
    R_M:= \frac{\|\alpha_M/\pi\|_2^2-1}{\|\alpha_M/\pi\|_2^2} \, .
\end{equation}

The quantity $R_M$ appears in \cite{AldousBrown1992}, as the intermediate quantity that they denote by $p_1$, but it appears that they did not observe that it takes the explicit form $\frac{\|\alpha_M/\pi\|_2^2-1}{\|\alpha_M/\pi\|_2^2}$. We can therefore, up to the value $\frac{\|\alpha_M/\pi\|_2^2-1}{\|\alpha_M/\pi\|_2^2}$ of $R_M$ and the removal of the irreducibility assumption, attribute the following theorem to Aldous and Brown.
\begin{thm}
\label{thm:ABrefinementintro} Let $X=(X_t)_{t\geq 0}$ be an irreducible reversible continuous-time Markov chain on a finite state space $S$, and denote its  stationary distribution by $\pi$. Let $\eset \ne A\subsetneq S$ and $M \in \MaxIrr_X(A^c)$. 
Then
\begin{equation}
\label{e:ABrefinement1}
    \pi(A) \le 1 - \frac{\mathbb{P}_{\pi}(T_A>t)}{\mathbb{P}_{\alpha_M}(T_A>t)} \le R_M \, ,
\end{equation}
for all $t\geq 0$, and
\begin{equation}\label{e:ABrefinement4}
\pi(A) \le 1-\frac{\mathbb{E}_{\pi}[T_A]}{\mathbb{E}_{\alpha_M}[T_A]}  \le R_M \, .
\end{equation}
\end{thm}
Theorem \ref{thm:ABrefinementintro} is better than Theorem \ref{T:ABintro} since we always have $R_M \leq \trel/\bbE_{\alpha_M} \cg T_A \cd$ by \cite[Lemma 10]{AldousBrown1992}.
We prove a modified variant of this result as Lemma \ref{lem:cheapboundon2norm}. However in general the bound $\trel/\bbE_{\alpha_M}\cg T_A \cd$ is far from being sharp: using Theorem \ref{thm:ABrefinementintro} together with bounds on $R_M$ leads to more precise results, as we will see.

\begin{remark}\label{rem: discrete and continuous}
We focus on the continuous-time setup. The analysis can easily be extended to the discrete-time setup with the following two minor modifications.
\begin{enumerate}
    \item Terms of the form $\exp(-(1-\gamma_i)t)$ need to be replaced with $\gamma_i^t$.
    \item The lower bounds on the probability that $T_{A}>t$ apply for even $t$ but could potentially fail for odd $t$; however if all eigenvalues of the transition matrix of the chain killed at $A$ are non-negative (e.g., if $\min_{x\in S}P(x,x) \ge 1/2$), they will hold for all integer $t \geq 0$.
    \item The identity $\bbP_\alpha(T_A>t) = \exp(-t/\bbE_\alpha T_A)$ has to be replaced by $(1-1/\bbE_\alpha T_A)^t$, where $M\in \MaxIrr_X(A^c)$ and $\alpha = \alpha_M$, but the hitting time $\bbE_\alpha T_A$ is unchanged.
\end{enumerate}   
\end{remark}

\subsection{Hitting times and mixtures of exponentials}

We now present a simple proof of the well-known fact that for continuous-time reversible chains, the hitting time of a set $A$ (such that $\eset \ne A \subsetneq S$) starting from the stationary distribution $\pi$ is a mixture of exponentials. In other words, its law is completely monotone. 

As before we write $B = A^c$, and \textbf{we do not assume that $B$ is $X$-irreducible.} We already identified $P_B$ as an operator on $C_0(B)$. Let us now do the same for $e^{tQ_B}$. For $f\in C_0(B)$ and $x\in S$, we have 
\begin{equation}\label{eq: identity expectation exponential Q B}
    (e^{tQ_B}f)(x)=\sum_{b\in B} \cg e^{tQ_B} \cd(x,b)f(b)=\E_x[f(X_{t}) \mathbf{1}\{T_{A}>t \}] \, .
\end{equation}
Let us decompose $1_B$ in the $\langle \cdot, \cdot \rangle_{\pi}$-orthonormal basis $(f_1,\ldots,f_m)$ basis of $C_0(B)$: there exist $c_1,\ldots, c_m \in \bbR$ such that
\begin{equation}
    1_B=\sum_{i=1}^m c_{i}f_i.
\end{equation} 
Moreover, for $1\leq i\leq m$, by orthonormality and since $f_i$ is supported on $B$, we have
\begin{equation}\label{e: c_i expectation form}
    c_i = \langle 1_B, f_i \rangle_{\pi} = \sum_{b\in B} \pi(b) f_i(b) = \sum_{x\in S} \pi(x) f_i(x) = \mathbb{E}_{\pi}[f_i] \, .
\end{equation}
Then, for $t\geq 0$ we have the identity
\begin{equation}
    e^{tQ_B} 1_{B}=\sum_{i=1}^m c_{i}e^{-t(1-\gamma_i)} f_i \, .
\end{equation}
Moreover, applying \eqref{eq: identity expectation exponential Q B} with $f = 1_B$ gives
\begin{equation}
\label{e:mixture0}
    \P_\pi[T_{A}>t] = \sum_{x\in S} \pi(x) \bbE_x \cg \mathbf{1} \ag T_{A}>t \ad \cd = \sum_{b\in B} \pi(b) (e^{tQ_B}1_B)(b) = \langle  e^{tQ_B} 1_{B},1_{B} \rangle_{\pi} \, .
\end{equation}
Therefore by orthogonality we obtain the identity
\begin{equation}
\label{e:mixture1}
\begin{split}
   \P_\pi[T_{A}>t] =\sum_{i=1}^m c_{i}^{2}e^{-t(1-\gamma_i)} \, .
\end{split}
\end{equation}
Equation \eqref{e:mixture1} shall play a crucial role in the proof of Theorem \ref{thm:ABrefinementintro} and in our refinement of the Aldous--Brown approximations.
Finally, applying \eqref{e:mixture1} with $t= 0$ gives
\begin{equation}\label{e:identity sum c i square}
    \sum_{i=1}^m c_i^2 = \bbP_\pi[T_A>0] = 1-\pi(A) = \pi(B) \, .
\end{equation}

\subsection{Proof of Theorem \ref{thm:ABrefinementintro}}
Recall that $B = A^c$, and that by Remark \ref{rem: unique maximal eigenvalue up to adding laziness locally} we may assume that there is a unique component $M\in \MaxIrr_X(B)$. A simple but key new observation is the following.

\begin{lemma}\label{lem: new identity c 1 square}
    We have $c_1^2 = 1-R_M$.
\end{lemma}
\begin{proof}
    By \eqref{e: c_i expectation form} we have $c_1 = \bbE_\pi[f_1]$. Moreover applying \eqref{e: f_1 is alpha rescaled} to $M$ and $M^c$ (in place of $B$ and $A$, since in \eqref{e: f_1 is alpha rescaled} the set $B$ was assumed to be irreducible), we get that $f_1=\frac{\alpha_M/\pi}{\|\alpha_M/\pi\|_2}$. We therefore get
\begin{equation}
    c_1 = \bbE_\pi \cg f_1 \cd = \frac{1}{\|\alpha_M/\pi\|_2}  \bbE_\pi \cg \alpha_M/\pi \cd =  \frac{1}{\|\alpha_M/\pi\|_2} \, ,
\end{equation}
and therefore $c_1^2 = 1-R_M$ by definition of $R_M$.
\end{proof}
Let us now prove Theorem \ref{thm:ABrefinementintro}.

\begin{proof}[Proof of Theorem \ref{thm:ABrefinementintro}]
    Let $t\geq 0$. By \eqref{e:mixture1} and since $\gamma_i\leq \gamma_1$ for $1\leq i\leq m$, we have 
    \begin{equation}
        c_1^2 e^{-t(1-\gamma_1)} \leq  \P_\pi[T_{A}>t] = \sum_{i=1}^m c_{i}^{2}e^{-t(1-\gamma_i)} \leq \pg \sum_{i=1}^m c_{i}^{2}\pd e^{-t(1-\gamma_1)} \, .
    \end{equation}
    
Moreover, $c_1^2 = 1-R_M$ by Lemma \ref{lem: new identity c 1 square} and $\sum_{i=1}^m c_{i}^{2} = 1-\pi(A)$ by \eqref{e:identity sum c i square}. Recalling \eqref{e: identity lambda gamma 1} and Lemma \ref{lem: QS reminders} we can rewrite the previous display as
\begin{equation}\label{e: eq inter proof theorem AB R probabilities}
    (1-R_M)\bbP_{\alpha_M}[T_A >t] \leq \P_\pi[T_{A}>t] \leq (1-\pi(A))\bbP_{\alpha_M}[T_A >t] \, ,
\end{equation}
which is equivalent to \eqref{e:ABrefinement1}.
Finally, integrating \eqref{e: eq inter proof theorem AB R probabilities} from 0 to $\infty$ gives 
\begin{equation}\label{e: eq inter proof theorem AB R expectations}
    (1-R_M)\bbE_{\alpha_M}[T_A] \leq \bbE_\pi[T_{A}] \leq (1-\pi(A))\bbE_{\alpha_M}[T_A] \, ,
\end{equation}
which is equivalent to \eqref{e:ABrefinement4}. This concludes the proof.
\end{proof}

\subsection{A simple deduction of (a variant of) Theorem \ref{T:ABintro} from Aldous and Brown}
For the sake of completeness, we prove a variant (see Remark \ref{rem: essentiall Lemma 10 b of AB} below) of a lemma of Aldous and Brown \cite[Lemma 10 part (b)]{AldousBrown1992}, which in conjunction with Theorem \ref{thm:ABrefinementintro} immediately implies Theorem \ref{T:ABintro}. We note that while the statement of the next lemma is essentially identical to that in \cite{AldousBrown1992}, its derivation in \cite{AldousBrown1992} is more complicated than the one presented here.
\begin{lemma}
\label{lem:cheapboundon2norm}
Let $X=(X_t)_{t\geq 0}$ be an irreducible reversible continuous-time Markov chain on a finite state space $S$, and denote its stationary distribution by $\pi$. Let $\eset \ne A\subsetneq S$ and $M \in \MaxIrr_X(A^c)$. 
Then
\begin{equation*}
    R_M \le \frac{\trel}{\mathbb{E}_{\alpha_M}[T_A]} \, . 
\end{equation*}
\end{lemma}
\begin{proof}
First, by Remark \ref{rem: unique maximal eigenvalue up to adding laziness locally} we may assume that there is a unique set $M\in \MaxIrr_X(A^c)$.

Since $f_1=\frac{\alpha_M/\pi}{\|\alpha_M/\pi\|_2}$ is supported on $M$ (and hence on $B=A^c$) we get that $((I-P) f_{1})(x)f_{1}(x)=\lambda f_1(x)^2$ for all $x\in S$, where we recall that $\lambda = 1-\gamma_1(B) = 1/\bbE_{\alpha_M}\cg T_A \cd$. This, together with the fact that $\|f_{1}\|_2=1$ gives that
\begin{equation}
    \frac{1}{\mathbb{E}_{\alpha_M}[T_{A}]}=\lambda=\frac{\langle (I-P) f_{1},f_{1} \rangle_{\pi}}{\|f_1\|_2^2}=\mathrm{Var}_{\pi}f_1 \frac{\langle (I-P) f_{1},f_{1} \rangle_{\pi}}{\mathrm{Var}_{\pi}f_1}.
\end{equation}
Using the extremal characterisation of the relaxation-time (see, e.g., Theorem 3.1 in \cite{BerestyckiNotesMixingTimes}; note that the proof of that result and identities works for substochastic $\pi-$reversible transition matrices), this implies that  $  \frac{\langle (I-P) f_{1},f_{1} \rangle_{\pi}}{\mathrm{Var}_{\pi}f_1} \ge 1/\trel$, and observing that 
$$\mathrm{Var}_{\pi}f_1=\frac{\|\alpha_M/\pi-1\|_2^2}{\|\alpha_M/\pi\|_2^2}=\frac{\|\alpha_M/\pi\|_2^2-1}{\|\alpha_M/\pi\|_2^2} = R_M$$ 
concludes the proof.
\end{proof}

\begin{remark}\label{rem: essentiall Lemma 10 b of AB}
    The bound proved in \cite[Lemma 10 part (b)]{AldousBrown1992} is $\|\alpha_M/\pi\|_2^2 \leq (1-\trel/\bbE_\pi [T_A])^{-1}$. In other words, since $R_M  = 1 - 1/\|\alpha_M/\pi\|_2^2$, their bound can be rewritten as
\begin{equation}\label{eq: bound AB Lemma 10 b rewritten with R M}
    R_M  \leq \frac{\trel}{\bbE_\pi [T_A]}. 
\end{equation}    
Since $\bbE_\pi [T_A] \leq \bbE_{\alpha_M}[T_A]$, the bound from Lemma \ref{lem:cheapboundon2norm} is better. However,
    these bounds are useful only if $\trel =o(\bbE_\pi T_A)$, in which case $\bbE_{\alpha_M} [T_A] = (1+o(1))\bbE_\pi [T_A]$, so the bound \eqref{eq: bound AB Lemma 10 b rewritten with R M} is as useful as Lemma \ref{lem:cheapboundon2norm} in practice.
\end{remark}

\section{Improved lower bound in the Aldous--Brown approximation}

\subsection{Interlacing relaxation times}
Let $\eset \ne A \subsetneq S$ and set $B = A^c$. Recall that we do not assume that the restriction of the chain to $B$ is irreducible, and that we may assume that $\MaxIrr_X(B)$ is a singleton $\ag M \ad$ by Remark \ref{rem: unique maximal eigenvalue up to adding laziness locally}. We also write $\pi_A(\cdot) = \pi(\cdot) / \pi(A)$.

We recall the construction of the auxiliary Markov chain in which $A = B^c$ is collapsed into a single point. Its state space is $B\cup \{A \}$. Its transitions are given by $K(x,y)=P(x,y)$,  $K(x,\{A\})=P(x,A)$,   $K(\{A\},x)=\sum_{a \in A}\pi_A(a)P(a,x)$ for $x,y \in B$ and $K(\{A\},\{A\})=\sum_{a \in A}\pi_A(a)P(a,A)$.  This is a reversible chain w.r.t.\ the distribution $\hat \pi $ given by $\hat \pi(\{A\})=\pi(A)$ and $\hat \pi(x)=\pi(x)$ for all $x \in B$.  Denote the relaxation time of $K$ by $\trel(K)$.

\begin{lemma}[Interlacing relaxation times]
\label{lem:interlacing}
We have that
\begin{equation}
\label{e:interlacing}
\frac{1}{1-\gamma_2(B)} \le \trel(K) \le \trel .
\end{equation}
\end{lemma}
\begin{proof}
Denote the second largest eigenvalue of $K$ by $\hat \lambda$. By the extremal characterisation of the second largest eigenvalue (e.g.\ \cite[Remark 13.8]{LivreLevinPeres2019MarkovChainsAndMixingTimesSecondEdition}) we get that
\begin{equation}
    \frac{1}{\trel}=\min_{f:\mathrm{Var}_{\pi}f \neq 0}\frac{\langle (I-P) f,f \rangle_{\pi}}{\mathrm{Var}_{\pi}f} \le \min_{\substack{f:\mathrm{Var}_{\pi}f \neq 0 \\ f(a)=f(b) \text{ for all }a,b \in A}}\frac{\langle (I-P) f,f \rangle_{\pi}}{\mathrm{Var}_{\pi}f} \, .
\end{equation}
Observe that for $f$ such that $f(a)=f(b)$ for all $a,b \in A$ we have that $\frac{\langle (I-P) f,f \rangle_{\pi}}{\mathrm{Var}_{\pi}f}=\frac{\langle (I-K) \hat f,\hat f \rangle_{\pi}}{\mathrm{Var}_{\hat \pi}\hat f}$, where $\hat f(x)=f(x)$ for all $x \in B$ and $\hat f(\{A\})=f(a)$ for $a \in A$. Hence the r.h.s.\ of the last display equals $1-\hat \lambda=1/\trel(K)$. To conclude the proof it remains to show that $\hat \lambda \ge \gamma_2 $. Observe that  $\widetilde{P}_B$ is obtained from $K$ by deleting the row and column corresponding to $\{A\}$. It follows from the interlacing eigenvalues theorem that indeed $\hat \lambda \ge \gamma_2$, as desired.
\end{proof}

\subsection{Sharpness of Theorem \ref{thm:ABrefinementintro}.}

We check that Theorem \ref{thm:ABrefinementintro} is essentially optimal, in that $R_M$ is both an upper bound \emph{and} a lower bound for the quantities in that theorem, fairly generally. 

To see this, we consider the contribution of the first eigenvalue separately, using Lemma \ref{lem: new identity c 1 square}, to make the contribution of the first term of the sum in \eqref{e:mixture1} explicit. 
This gives, recalling that $1-\gamma_1 = \lambda = 1/\bbE_{\alpha_M}\cg T_A\cd$, for $t\geq 0$,
\begin{equation}
\label{e:mixture2}
    \P_\pi[T_{A}>t] = (1-R_M)e^{-t/\mathbb{E}_{\alpha_M}[T_A]}+\sum_{i=2}^m c_{i}^{2}e^{-t(1-\gamma_i)} \, .
\end{equation}
Let $t\geq 0$. Since all terms in the sum of \eqref{e:mixture2} are non-negative, $\gamma_2 \geq \gamma_i$ for $2\leq i\leq m$, and $1-\gamma_2 \geq 1/\trel$ by Lemma \ref{lem:interlacing}, we have
\begin{equation}
\label{e:mixture2bounded}
\begin{split}
     (1-R_M)\bbP_{\alpha_M}[T_A>t] & \leq \P_{\pi}[T_{A}>t] \\
     & \leq (1-R_M)\bbP_{\alpha_M}[T_A>t]+\pg \sum_{i=2}^m c_{i}^{2} \pd e^{-t(1-\gamma_2)} \\
     & \leq (1-R_M)\bbP_{\alpha_M}[T_A>t]+ ((1-\pi(A))-(1-R_M)) e^{-t/\trel} \, .
\end{split}
\end{equation}
Dividing by $\bbP_{\alpha_M}[T_A>t]$, subtracting $1$, multiplying by $-1$, and factorising by $R_M$, this can be rewritten as
\begin{equation}
\label{eq: mu pi tail comparison probabilities}
\begin{split}
     R_M\pg 1- \pg 1 - \frac{\pi(A)}{R_M}\pd \frac{e^{-t/\trel}}{\bbP_{\alpha_M}[T_A>t]}\pd \leq 1- \frac{\P_{\pi}[T_{A}>t]}{\bbP_{\alpha_M}[T_A>t]} 
     \leq R_M \, .
\end{split}
\end{equation}
Integrating \eqref{e:mixture2bounded}, we get 
\begin{equation}
\label{e:mixture2boundedintegrated}
     (1-R_M)\bbE_{\alpha_M}[T_A] \leq \bbE_{\pi}[T_{A}]
     \leq (1-R_M)\bbE_{\alpha_M}[T_A]+ (R_M-\pi(A))\trel \, ,
\end{equation}
and reordering the terms leads to
\begin{equation}\label{eq: mu pi tail comparison expectations}
     R_ M \pg 1 - \pg 1 - \frac{\pi(A)}{R_M}\pd\frac{\trel}{\bbE_{\alpha_M}[T_A]} \pd \leq 1- \frac{\bbE_{\pi}[T_A]}{\bbE_{\alpha_M}[T_A]}\leq R_M \, .
\end{equation}
Thus provided that $A$ is not too large, and $\trel = o( \E_{\alpha_M} (T_A))$, we see that $R_M$ is the optimal bound in \eqref{e:ABrefinement1} of Theorem \ref{thm:ABrefinementintro}. 

\begin{remark}
    Let us examine the lower bound \eqref{eq: mu pi tail comparison probabilities} above in the concrete case of tori $(\mathbb{Z}/ \ell \mathbb{Z})^d$ of sidelength $\ell\ge 2$ in dimension $d \ge 2$, and for sets $A$ of uniformly bounded size. Then $\trel \asymp \ell^2 = o(\bbE_{\alpha_M} T_A)$, so as soon as $t/\ell^2 \to \infty$, we have $\frac{e^{-t/\trel}}{\bbP_{\alpha_M}[T_A>t]} = o(1)$, and we therefore have the asymptotic equivalence (as $\ell \to \infty$)
    \begin{equation}
         1- \frac{\bbP_{\pi}[T_A>t]}{\bbP_{\alpha_M}[T_A>t]} \sim R_M \, .
    \end{equation}
\end{remark}

Any improvement on the results of Aldous and Brown \cite{AldousBrown1992} must therefore come from an improved bound on $R_M$.

\section{A general bound for reversible Markov chains}
\label{s:QS theory and AB}

We already saw that Theorem \ref{thm:ABrefinementintro} is at least as good as the bound of Aldous and Brown which we recalled as Theorem \ref{T:ABintro}. 
To make use of Theorem \ref{thm:ABrefinementintro} we need to prove a better upper bound on $R_M-\pi(A)$ than $\trel/\mathbb{E}_{\alpha_M}[T_A]$.
To this purpose, we need to understand more precisely the hitting time of $A$ starting from any state $x$, and we will also need to introduce an auxiliary time, $\tmed$.

As before, $X = (X_t)_{t\geq 0}$ is an irreducible reversible continuous-time Markov chain on a finite state space $S$ with stationary distribution $\pi$, $\eset \ne A\subsetneq S$, $B = A^c$, we may assume that $|\MaxIrr_X(B)| = 1$, and $M\in \MaxIrr_X(B)$.

\subsection{Hitting times when the walk starts at a given vertex}

Recall that $f_1=\frac{\alpha_M/\pi}{\|\alpha_M/\pi\|_2},f_2,\ldots,f_m$ is an orthonormal basis of $C_0(B)$ and that $e^{tQ_B}f_i = e^{-\lambda_i t} f_i$ for all $1\leq i\leq m$ and $t\geq 0$. 

\begin{lemma}\label{lem:decompositionfromx}
Let $\eset \ne A\subsetneq S$ and set $B = A^c$ (which is not assumed to be $X$-irreducible). For every $x\in B$ and $t\geq 0$, we have
    \begin{equation}
    \label{e:decompositionfromx}
    \P_x[T_{A}>t]=\sum_{i=1}^m c_{i}f_i(x)e^{-\lambda_it} \, .
    \end{equation}
    \end{lemma}
\begin{proof}
     Let $x\in B$ and $t\geq 0$. For $y\in B$, we have
\begin{equation}
    \P_x[X_t=y, T_{A}>t] = [e^{t Q_B}](x,y) = \langle e^{tQ_B}1_y, 1_x\rangle = \left\langle e^{tQ_B}1_y, \,  \frac{1_x}{\pi(x)}\right\rangle_{\pi} \, .
\end{equation}
Summing over all $y\in B$ and using the decomposition $1_B=\sum_{i=1}^m c_{i}f_i$, we obtain
\begin{equation}
    \P_x[T_{A}>t] =  \left\langle e^{tQ_B}1_B, \,  \frac{1_x}{\pi(x)} \right\rangle_{\pi} = \left\langle e^{tQ_B}\sum_{i=1}^m c_{i}f_i, \,  \frac{1_x}{\pi(x)} \right \rangle_{\pi} = \left\langle \sum_{i=1}^m c_{i}e^{-\lambda_it}f_i, \, \frac{1_x}{\pi(x)}\right\rangle_{\pi} \, .
\end{equation}
We deduce, developing the inner product, that
\begin{equation}
\begin{split}
     \P_x[T_{A}>t] = \sum_{i=1}^m c_i e^{-\lambda_it} \left\langle f_i, \frac{1_x}{\pi(x)}\right\rangle_{\pi} & = \sum_{i=1}^m c_i e^{-\lambda_it} \sum_{z\in S} \pi(z) f_i(z) \frac{1_x(z)}{\pi(x)} \\
     & = \sum_{i=1}^m c_i e^{-\lambda_it} f_i(x) \, ,
\end{split}
\end{equation}
which concludes the proof.
\end{proof}

\begin{lemma}\label{lem:identityhittingprobabilitiesfullsum}
Let $\eset \ne A\subsetneq S$ and set $B = A^c$ (which is not assumed to be $X$-irreducible). For every $t\geq 0$, we have
\begin{equation}
    \sum_{x\in B} \pi(x) (\bbP_x(T_A\leq t))^2 = \sum_{i= 1}^m c_i^2(1-e^{-\lambda_i t})^2 \, .
\end{equation}
\end{lemma}
\begin{proof}
Let $t\geq 0$. By Lemma \ref{lem:decompositionfromx} we have for $x\in B$, using that $\sum_{i=1}^m c_if_i(x) = 1_B(x) = 1$ for the last equality,
\begin{equation}
  g_t(x) := \bbP_x[T_A\leq t] = 1- \P_x[T_{A}>t] = 1 - \sum_{i=1}^m c_if_i(x) e^{-\lambda_i t} = \sum_{i=1}^m c_if_i(x) (1-e^{-\lambda_i t}) \, ,
\end{equation}
that is, $g_t = \sum_{i=1}^m c_i (1-e^{-\lambda_i t})f_i$.
By orthonormality of the family $(f_i)_{1\leq i\leq m}$, we conclude that
\begin{equation}
    \sum_{x\in B} \pi(x) (\bbP_x[T_A\leq t])^2 = \langle g_t, g_t \rangle_\pi = \sum_{i=1}^m c_i^2 (1-e^{-\lambda_i t})^2 \, . \qedhere
\end{equation}
\end{proof}

\subsection{General bound with an auxiliary time}

We introduce the following auxiliary time, which will prove to be very useful: let $\tmed = \tmed(B)$ be the time (which exists and is unique by monotonicity) satisfying
\begin{equation}\label{e: definition of t med}
    \sum_{i=2}^m c_i^2(1-e^{-\lambda_i\tmed})^2= \frac 12 \sum_{i=2}^m c_i^2 \, .
\end{equation}

\begin{prop}[Bound on $R_M-\pi(A)$]
\label{prop:boundusingdsepQSa}
Let $\eset \ne A\subsetneq S$ and set $B = A^c$ (which is not assumed to be $X$-irreducible).
We have
\begin{equation}
\label{e:boundusingdsepQS2}
        R_M-\pi(A) = 2 \sum_{x\in B} \pi(x) \left(   \mathbb{P}_x[T_A \le \tmed ]  \right)^2 -2  (1-R_M)(1-e^{-\lambda_1 \tmed })^{2} \, .
\end{equation}
In particular, we have the following bound
\begin{equation}
\label{e:boundusingdsepQS2ineq}
        R_M-\pi(A) \leq 2 \sum_{x\in B} \pi(x) \left(   \mathbb{P}_x[T_A \le \tmed ]  \right)^2 \, .
\end{equation}
\end{prop}
\begin{proof}
By Lemma \ref{lem:identityhittingprobabilitiesfullsum} and the definition of $\tmed$, we have
   \begin{equation}
    \sum_{x\in B} \pi(x) (\bbP_x(T_A\leq \tmed))^2
    = c_1^2(1-e^{-\lambda_1 \tmed})^2 + \frac{1}{2} \sum_{i=2}^m c_i^2 \, .
\end{equation}
Since $c_1^2 = 1-R_M$ and $\sum_{i=2}^m c_i^2 = R_M-\pi(A)$, this can be rewritten as
\begin{equation}
    \sum_{x\in B} \pi(x) (\bbP_x(T_A\leq \tmed))^2 = (1-R_M)(1-e^{-\lambda_1 \tmed})^2 + \frac{1}{2} (R_M-\pi(A)) \, ,
\end{equation}
and reordering the terms gives the desired result.
\end{proof}

\subsection{A general bound via hitting before the relaxation time}

The auxiliary time $\tmed$ does not seem to be easy to estimate. However, we can always bound it by twice the relaxation time, as we now prove.

\begin{lemma}\label{lem:bound tmed trel}
    We have $e^{-\tmed/\trel} \geq 1-1/\sqrt{2}$, and in particular $\tmed \leq 2\trel$.
\end{lemma}
\begin{proof}
    First, for $2\leq i\leq m$ we have $\gamma_i \leq \gamma_2$, so  $\lambda_i \geq  \lambda_2$. Moreover,  we have $\lambda_2 = 1-\gamma_2 \ge 1/\trel$ by Lemma \ref{lem:interlacing}. We deduce that for each $2\leq i\leq m$, we have $\lambda_i \geq 1/\trel$, and therefore that $(1-e^{-\lambda_i\tmed})^2 \geq (1-e^{-\tmed/\trel})^2$. Applying this to each term of the sum in \eqref{e: definition of t med}, we obtain $1/2 \geq  (1-e^{-\tmed/\trel})^2$, which can be rewritten as $e^{-\tmed/\trel} \geq 1-1/\sqrt{2}$. Therefore we have $\tmed \leq -\log(1-1/\sqrt{2})\trel \leq (5/4) \trel$.
\end{proof}

We can now prove Theorem  \ref{thm: general bound bound with relaxation time intro}.

\begin{proof}[Proof of Theorem \ref{thm: general bound bound with relaxation time intro}]
Recall from Remark \ref{rem: unique maximal eigenvalue up to adding laziness locally} that we may assume that $\MaxIrr_X(B) = \ag M \ad$.
    By Lemma \ref{lem:bound tmed trel}, for each $x\in B$ we have $\bbP_x[T_A\leq\tmed] \leq \bbP_x[T_A\leq 2\trel]$. The result then follows from Theorem \ref{thm:ABrefinementintro} and Proposition \ref{prop:boundusingdsepQSa}.
\end{proof}

\section{Specialisation of the results for vertex-transitive graphs}\label{s: specialisation of the results for vertex-transitive graphs}
In Theorem \ref{thm: general bound bound with relaxation time intro}, we proved a practical bound written as a function of hitting times before the relaxation time. The goal of this section is to bound the quantity 
\begin{equation}
    \sum_{x\in S} \pi(x) \left(   \mathbb{P}_x[T_A \le t_2 ]  \right)^2,
\end{equation}
for vertex transitive graphs, where $t_2 := 2\trel$.

\subsection{Reduction of the statement to a singleton}
Let $\Gamma = (V,E)$ be a finite connected vertex-transitive graph. The Markov chain we consider in this section is the (rate-1 continuous time) simple random walk on $\Gamma$. Let $o\in V$ be any vertex. Note that the statespace is now $S=V$ and that the stationary measure is uniform, that is, $\pi(x) = 1/|V|$ for every $x\in V$.
\begin{lemma}\label{lem:bound A square CS}
    Let $A \subset V$ and $t\geq 0$. We have
    \begin{equation}
        \sum_{x\in V}\pi(x)\bbP_x(T_A \leq t)^2 \leq |A|^2  \sum_{x\in V}\pi(x)\bbP_x(T_o\leq t)^2.
    \end{equation}
\end{lemma}
\begin{proof}
    For $x\in V$, by the Cauchy--Schwarz inequality we have
    \begin{equation}
        \bbP_x(T_A \leq t)^2\leq \pg \sum_{a\in A} \bbP_x(T_a \leq t)\pd^2 \leq  |A|\sum_{a\in A} \bbP_x(T_a \leq t)^2.
    \end{equation}
Permuting the sums and using transitivity, we obtain
    \begin{equation}\begin{split}
         \sum_{x\in V}\pi(x)\bbP_x(T_A \leq t)^2   \leq |A| \sum_{a\in A} \sum_{x\in V}\pi(x) \bbP_x(T_a \leq t)^2 & = |A| \sum_{a\in A} \sum_{x\in V}\pi(x) \bbP_x(T_o \leq t)^2 \\
         & = |A|^2 \sum_{x\in V}\pi(x) \bbP_x(T_o \leq t)^2,
    \end{split}
    \end{equation}
as desired.
\end{proof}

\subsection{Killing the chain at an exponential time}
Our goal is now to bound $\sum_{x\in V}\pi(x) \bbP_x(T_o \leq t_2)^2$, where we recall that $t_2 = 2\trel$.
For $t\geq 0$, denote the local time at $o$ by time $t$ by $L_o(t)= \int_{0}^t 1_{X_s = o} \mathrm{d}s$. For $x,y\in V$ and $t\geq 0$, set $p_t(x,y):= \bbP_x(X_t = y)$. For any $x\in V$ and $t\geq 0$ we therefore have 
\begin{equation}\label{eq: expected local times as integrals}
  \bbE_x L_o(t) = \int_0^t p_s(x,o) \mathrm{d}s.  
\end{equation} 
Moreover for any $x\in V$ and $t\geq 0$ we have 
\begin{equation}
   \bbP_x[T_o \leq t] = \frac{\bbE_x[L_o(t)]}{\bbE_x [L_o(t)\mid L_o(t)>0]}. 
\end{equation}
Upper bounding the numerator can be done using the integral form \eqref{eq: expected local times as integrals}. For the denominator, we would like to lower bound $\bbE_x [L_o(t)\mid L_o(t)>0]$ by $\bbE_o [L_o(t/2)]$, perhaps up to a constant, and then use the integral form also for the denominator. However such a bound does not hold in general. (For example in discrete time, on a segment of length $n$, at time $t=n$, we have $\bbE_n [L_0(n)\mid L_0(n)>0] = 1$, while $\bbE_0 [L_0(n/2)] \asymp \sqrt{n}$.)

\medskip

To go around this, we consider a version of the chain that is killed at rate $1/t_2$. Let $\tau$ be an exponential variable of parameter $1/t_2$. Observe that since $\bbP(\tau \geq t_2) =e^{-1}$, we have
\begin{equation}\label{e:bound when t 2 is replaced by tau}
     \sum_{x\in V}\pi(x)\bbP_x(T_A \leq t_2)^2 \leq e^2 \sum_{x\in V}\pi(x)\bbP_x(T_A \leq \tau)^2. 
\end{equation}
Replacing $t_2$ by $\tau$, we therefore only lose a multiplicative constant.
Let $N := L_o(\tau)$ be the local time at $o$ before the chain is killed. For $x,y\in V$ and $s\geq 0$ we have $\mathbb{P}_x[X_s=y,\tau>s]=p_s(x,y)e^{-\frac{s}{t_{1}}}$, and hence for $x\in V$,
\begin{equation}
    \mathrm{E}_x[N]=\int_0^{\infty}p_s(x,o)e^{-\frac{s}{t_{1}}}\mathrm{d}s.
\end{equation}
\begin{lemma}\label{lem:memoryless N identity} For any $x\in V$, we have
 \begin{equation}
    \mathbb{P}_x[T_o \le \tau] = \frac{\int_0^{\infty}p_s(x,o)e^{-s/t_2}\mathrm{d}s}{\int_0^{\infty}p_s(o,o)e^{-s/t_2} \mathrm{d}s}.
\end{equation}
\end{lemma}
\begin{proof}
Let $x\in V$. First, we have
\begin{equation}
    \mathbb{P}_x[T_o \le \tau] =  \bbP_x[N \mid N>0]=\frac{\mathrm{E}_x[N]}{\mathrm{E}_x[N \mid N>0]} \, .
\end{equation}
  Moreover, by the memoryless property of the exponential distribution, we have $\mathrm{E}_x[N \mid N>0]=\mathrm{E}_{o}[N]$. 
We conclude that
\begin{equation}
     \mathbb{P}_x[T_o \le \tau]=\frac{\mathrm{E}_x[N]}{\mathrm{E}_o[N]} = \frac{\int_0^{\infty}p_s(x,o)e^{-s/t_2}\mathrm{d}s}{\int_0^{\infty}p_s(o,o)e^{-s/t_2} \mathrm{d}s} \, . \qedhere
\end{equation}
\end{proof}
For $j\geq 0$ we set 
\begin{equation}
    I_j = \int_0^{\infty}s^jp_s(x,o)e^{-s/t_2}\mathrm{d}s \, .
\end{equation}

\begin{lemma}\label{lem: rewritting the sum of hitting times squared as a fraction of integrals}
We have
\begin{equation}
    \sum_{x\in V}\pi(x) (\mathbb{P}_x[T_o \le \tau])^2 = \frac{1}{|V|} \frac{I_1}{I_0^2} \, .
\end{equation}
\end{lemma}
\begin{proof}
     Since $\pi(x) = 1/|V|$ for every $x\in V$, by Lemma \ref{lem:memoryless N identity} we have
\begin{equation}
    |V| \pg \int_0^{\infty}p_s(x,o)e^{-s/t_2}\mathrm{d}s \pd^2\sum_{x\in V} \pi(x)(\mathbb{P}_x[T_o \le \tau])^2 = \sum_{x\in V} \pg \int_0^{\infty}p_s(x,o)e^{-s/t_2}\mathrm{d}s\pd^2.
\end{equation}
Moreover by reversibility we have $p_s(o,x) = p_s(x,o)$ for every $s\geq 0$ and $x\in V$. Therefore, permuting the sum and the integrals, we obtain
\begin{equation}
\begin{split}
     \sum_{x\in V} \pg \int_0^{\infty}p_s(x,o)e^{-s/t_2}\mathrm{d}s\pd^2 & = \sum_{x\in V} \int_0^{\infty}p_{s_1}(o,x)e^{-{s_1}/t_2}\mathrm{d}{s_1} \int_0^{\infty}p_{s_2}(x,o)e^{-{s_2}/t_2}\mathrm{d}{s_2} \\
    & =  \int_0^{\infty}\int_0^{\infty} e^{-{(s_1+s_2)}/t_2} \pg \sum_{x\in V} p_{s_1}(o,x)p_{s_2}(x,o) \pd \mathrm{d}{s_2}\mathrm{d}{s_1} \\
    & = \int_0^{\infty}\int_0^{\infty} e^{-{(s_1+s_2)}/t_2} p_{s_1+s_2}(o,o) \mathrm{d}{s_2}\mathrm{d}{s_1} \, .
\end{split} 
\end{equation}
Finally, doing the change of variables $s = s_1 + s_2$ and permuting integrals, we obtain
\begin{equation}
\begin{split}
     \sum_{x\in V} \pg \int_0^{\infty}p_s(x,o)e^{-s/t_2}\mathrm{d}s\pd^2 & = \int_0^{\infty}\int_0^{\infty} e^{-{s}/t_2} p_{s}(o,o) 1_{s\geq s_1}\mathrm{d}{s}\mathrm{d}{s_1} \\
     & = \int_{0}^\infty s\cdot p_s(o,o) e^{-s/t_2}\mathrm{d}s \, ,
\end{split}
\end{equation}
which concludes the proof.
\end{proof}

\subsection{Spectral decomposition of the killed chain}
Set $n = |V|$. Denote the transition matrix of the chain by $P$. Let $0 = \beta_1 < \beta_2 \leq \ldots \leq \beta_n$ be the eigenvalues of $I-P=: -Q$, and let $g_1,\ldots, g_n$ be an associated orthonormal basis of $\mathbb{R}^V$ with respect to $\langle \cdot,\cdot \rangle_{\pi}$. Then for $s\geq 0$ and $1\leq i\leq n$ we have $e^{sQ}g_i = e^{-s\beta_i} g_i$.
\begin{lemma}\label{lem:integrals rewritten with beta i}
    Recall that $n = |V|$, and let $j\geq 0$ be an integer. We have
    \begin{equation}
\label{e:stop at beta1 new}
\frac{n}{j!} I_j = t_2^{j+1}+ \sum_{i=2}^n (\beta_i + 1/t_2)^{-(j+1)} \, .
\end{equation}
\end{lemma}
\begin{proof}
    For $\theta>0$ we have $\int_0^\infty s^j e^{-s/\theta}\mathrm{d}s = j!\theta^{j+1}$. Applying this with $\theta = t_2$ and using transitivity, we get
    \begin{equation}
       nI_j - j!t_{2}^{j+1}
         = \int_0^{\infty} ns^j \left( p_s(o,o)-\pi(o)\right) e^{-\frac{s}{t_{2}}}\mathrm{d}s
         = \int_0^{\infty} s^j \sum_{x\in V} \left( p_s(x,x)-\pi(x)\right) e^{-\frac{s}{t_{2}}}\mathrm{d}s \, .
    \end{equation}
Moreover, for each $s\geq 0$, we have
\begin{equation}\label{e: identity heat kernel eigenvalues}
     \sum_{x\in V} \left( p_s(x,x)-\pi(x)\right) = \Tr \pg e^{Qs}-I/n\pd = \pg \Tr e^{Qs}\pd-1 = \pg \sum_{i=1}^n e^{-\beta_i s}\pd -1 = \sum_{i=2}^n e^{-\beta_i s} \, .
\end{equation}
Therefore, 
\begin{equation}
   nI_j - j!t_{2}^{j+1} = \sum_{i=2}^n\int_0^{\infty}  s^je^{-s \pg \beta_i +  \frac{1}{t_{2}}\pd}\mathrm{d}s = k!\sum_{i=2}^n (\beta_i + 1/t_2)^{-(k+1)} \, ,
\end{equation}
which concludes the proof. 
\end{proof}

\subsection{Bounds on the integrals}
We saw in Lemma \ref{lem: rewritting the sum of hitting times squared as a fraction of integrals} that $ \sum_{x\in V}\pi(x) (\mathbb{P}_x[T_o \le \tau])^2 = \frac{1}{|V|} \frac{I_1}{I_0^2}$, and rewrote $I_0$ and $I_1$ differently in Lemma \ref{lem:integrals rewritten with beta i}. We now want to find an upper bound on $I_1$ and a lower bound on $I_0$.

Let us start with $I_0$.

\begin{lemma}\label{lem:lower bound on I 0}
    We have
    \begin{equation}
        I_0 \geq \frac{2}{3}\frac{\bbE_\pi[T_o]}{n} \, .
    \end{equation}
\end{lemma}
\begin{proof}
    Recall that $2/t_2 = 1/\trel = \beta_2\leq \beta_3\leq \ldots \leq \beta_n$. It follows that for each $2\leq i \leq n$ we have $\beta_i + 1/t_2 \leq \frac{3}{2}\beta_i$, and therefore by Lemma \ref{lem:integrals rewritten with beta i} we have
    \begin{equation}\label{e:lower bound on I 0 eq inter}
       nI_0 \geq  nI_0 - t_2 \geq \frac{2}{3}\sum_{i=2}^n \frac{1}{\beta_i} \, .
    \end{equation}
Finally, by the eigentime identity (\cite[Proposition 3.13]{AldousFill}) and transitivity, we have
\begin{equation}
    \sum_{i=2}^n \beta_i^{-1}=\sum_{x}\pi(x)\mathbb{E}_{\pi}[T_x] = \mathbb{E}_{\pi}[T_o].
\end{equation}
Plugging this into \eqref{e:lower bound on I 0 eq inter} concludes the proof.
\end{proof}

Let us now bound $I_1$.
Recall that $D$ denotes the diameter of $\Gamma$ and $d$ is its degree. For $\rho>0$, denote the number of vertices in a ball of radius $\lfloor \rho \rfloor $ by $V(\rho)$. 
\begin{lemma}\label{lem:upper bound on I 1}
    We have
    \begin{equation}
        I_1 \leq 64 \pg  \frac{\trel^2}{n} + d^2\int_{0}^{\sqrt{\trel/d}}\frac{s^{3} \mathrm{d}s}{V(s)}   \pd.
    \end{equation}
\end{lemma}
\begin{proof}
First, by Lemma \ref{lem:integrals rewritten with beta i} we have
\begin{equation}
    I_1 -\frac{t_2^2}{n} = \frac{1}{n}\sum_{i=2}^n (\beta_i + 1/t_2)^{-2} \leq \frac{1}{n}\sum_{i=2}^n \beta_i^{-2} \, .
\end{equation}
Define the spectral measure by
\begin{equation}
    \mu := \frac{1}{n}\sum_{i=1}^n \delta_{\beta_i} \, .
\end{equation}
Recall that $\beta_1 = 0$, and define also 
\begin{equation}
    \nu := \mu - \frac{\delta_{\{0\}}}{n} = \frac{1}{n}\sum_{i=2}^n \delta_{\beta_i} \, , 
\end{equation}
which we emphasise is not a probability measure since its total mass is $1-1/n$.
Let $X \sim \mu$.
Doing the change of variables $r=1/\sqrt{u}$, we have
\begin{equation}
    \frac {1}{n}\sum_{i=2}^{n}\beta_i^{-2}  = \mathbb{E}[X^{-2} \mathbf{1}\{X>0 \} ]=\int_0^{\beta_2^{-2}}\P[u \le  X^{-2} < \infty]\mathrm{d}u =\int_{\beta_2}^{\infty}\frac{2}{r^3}\P[0 <  X \leq r]\mathrm{d}r \, .
\end{equation}
Using that $\mu(\pg 0, r\cd) = \mu(\cg 0, r \cd) - \frac{1}{n} = \nu  (\cg 0, r \cd)$ for $r>0$, we obtain
\begin{equation}
    \int_{\beta_2}^{\infty}\frac{2}{r^3}\P[0 <  X \leq r]\mathrm{d}r =  \int_{\beta_2}^{\infty} \frac{2}{r^3}\mu((0,r])\mathrm{d}r = \int_{\beta_2}^{\infty} \frac{2}{r^3}\nu([0,r])\mathrm{d}r\, .
\end{equation}
Now, by \cite[Theorem 1.7]{LyonsOveisGharan2018} (with their parameter $\alpha$ set to $\alpha = 1/2$), we have for each $r>0$
\begin{equation}
    \nu([0,r]) \le \frac{4}{V(\sqrt{(2d r)^{-1}})} \, .
\end{equation}
Substituting this bound and using the change of variables $s=\sqrt{(2d r)^{-1}}$ yields, recalling that $\beta_2 = 1/\trel$,
\begin{equation}
\begin{split}
     \int_{\beta_2}^{\infty}\frac{2}{r^3} \nu([0,r])\mathrm{d}r \leq \int_{0}^{(2d\beta_2)^{-1/2}} 2(2ds^2)^3 \frac{4}{V(s)}  \frac{2}{ds^3}\mathrm{d}s & = 64 d^{2} \int_{0}^{(2d\beta_2)^{-1/2}}\frac{s^{3} \mathrm{d}s}{V(s)} \\
     & \leq 64 d^{2} \int_{0}^{\sqrt{\trel/d}}\frac{s^{3} \mathrm{d}s}{V(s)} \, .
\end{split}
\end{equation}
We conclude that 
\begin{equation}
     I_1 \leq  \frac{(2\trel)^2}{n} + 64 d^{2} \int_{0}^{\sqrt{\trel/d}}\frac{s^{3} \mathrm{d}s}{V(s)}  \leq 64 \pg  \frac{\trel^2}{n} + d^2\int_{0}^{\sqrt{\trel/d}}\frac{s^{3} \mathrm{d}s}{V(s)}   \pd. \qedhere
\end{equation}
\end{proof}

\subsection{Combined and specialised bounds}
Thanks to our bounds on $I_0$ and $I_1$, we can now prove the general bound in function of the growth of the graph stated in Theorem \ref{thm: bounds QS for vertex transitive graphs intro with integrals and volumes}.

\begin{proof}[Proof of Theorem \ref{thm: bounds QS for vertex transitive graphs intro with integrals and volumes}]
     By \eqref{e:bound when t 2 is replaced by tau}, Lemma \ref{lem:bound A square CS}, (noting that $e^2\leq 8$), and Lemma \ref{lem: rewritting the sum of hitting times squared as a fraction of integrals}, we have
    \begin{equation}
         2\sum_{x\in S} \pi(x) \left(   \mathbb{P}_x[T_A \le 2\trel ]  \right)^2 \leq 16 |A|^2 \frac{I_1}{nI_0^2} = 16 |A|^2 \frac{nI_1}{(nI_0)^2}.
    \end{equation}
We conclude, using Lemmas \ref{lem:lower bound on I 0} and \ref{lem:upper bound on I 1}, that
\begin{equation}
\begin{split}
    16 |A|^2 \frac{1}{(nI_0)^2} nI_1 & \leq (16 |A|^2) \pg \frac{1}{\frac{2}{3} \bbE_\pi[T_o]} \pd^2 64 \pg  \frac{\trel^2}{n} + d^2\int_{0}^{\sqrt{\trel/d}}\frac{s^{3} \mathrm{d}s}{V(s)}   \pd \\
    & = C_0|A|^2 \pg \frac{\trel}{\bbE_\pi[T_o]} \pd^2 \pg 1 + \frac{d^2 n}{\trel^2}\int_{0}^{\sqrt{\trel/d}}\frac{s^{3} \mathrm{d}s}{V(s)}   \pd,
\end{split}
\end{equation}
where $C_0 = 16\cdot (3/2)^2 \cdot 64 = 2304$. 
\end{proof}

\begin{corollary}\label{cor: bounds for vertex transitive graphs with integrals and volumes}
    There exists a universal constant $C_0$ such that the following holds. Let $\Gamma = (V,E)$ be a finite connected vertex-transitive graph. Denote its degree by $d$ and its diameter by $D$. Let $\eset \ne A \subsetneq V$ be a subset of $V$ (whose complement $A^c$ is not necessarily connected in $\Gamma$).  We have 
\begin{equation}
    2\sum_{x\in S} \pi(x) \left(   \mathbb{P}_x[T_A \le 2\trel ]  \right)^2 \leq C_0|A|^2d^2 \pg \frac{D^2}{\bbE_\pi[T_o]} \pd^2 \pg 1 + \frac{n}{D^4}\int_{0}^{D}\frac{s^{3} \mathrm{d}s}{V(s)}   \pd.
\end{equation}
\end{corollary}
\begin{proof}
   By \cite[Theorem 1.7]{LyonsOveisGharan2018} we have $\beta_2 \ge \frac{1}{dD^2}$ (this uses transitivity) i.e.\ $\trel \leq dD^2$. Plugging this into Theorem \ref{thm: bounds QS for vertex transitive graphs intro with integrals and volumes} gives the desired result.
\end{proof}

The growth of finite vertex-transitive graphs was recently understood by Tessera and Tointon \cite[Proposition 6.1]{TesseraTointonIsop}.
In particular they proved the following.

For each integer $q\geq 1$, there exists a constant $c(q)$, such that for every finite (connected) vertex-transitive graph $\Gamma = (V,E)$ with size $|V| = D^qR$ for some $R\in [1, D)$, where $D$ is the diameter of $\Gamma$; for every $1\leq s \leq D$ the size $V(s)$ of the ball of radius $s$ satisfies
\begin{equation}\label{e:tesseratointon volume lower bound first}
    V(s) \ge \begin{cases}c(q) s^{q+1} & s \le R \\
c(q)R s^{q} & s > R  \\
\end{cases} \, ,
\end{equation}\label{e:tesseratointon volume lower bound second}
and, if $q\geq 5$,
\begin{equation}
    V(s)\geq c(5)s^5 \, .
\end{equation}

Using these growth bounds enables specialising Corollary \ref{cor: bounds for vertex transitive graphs with integrals and volumes}, as we stated in Theorem \ref{thm:main improvement AB transitive}. Let us prove this result. Recall from \eqref{e:def of beta Gamma} the definition of $\beta(\Gamma)$.

\begin{proof}[Proof of Theorem \ref{thm:main improvement AB transitive}]
    Set $c^* := \min(\min_{1\leq q \leq 5} c(q), 1)$. First assume that $q\geq 5$. Applying \eqref{e:tesseratointon volume lower bound second} we obtain
\begin{equation}
    \int_{0}^{D}\frac{s^{3} \mathrm{d}s}{V(s)} \leq 1 + \int_{1}^{D}\frac{s^{3} \mathrm{d}s}{c^*s^{5}} = 1 + \frac{1}{c^*}\int_{1}^{D}\frac{\mathrm{d}s}{s^{2}} \leq  \frac{2}{c^*} \, .
\end{equation}    
Moreover, since $n\geq D^4$, we have $1+\frac{n}{D^4} \int_{0}^{D}\frac{s^{3} \mathrm{d}s}{V(s)} \leq \frac{3}{c^*} \frac{n}{D^4}$, and $\mathbb{E}_{\pi}[T_o]=\sum_{i=2}^{n}\frac{1}{\beta_i} \ge \frac12 (n-1) \ge \frac n4$. Therefore,
\begin{equation}
     \pg \frac{D^2}{\bbE_\pi[T_o]} \pd^2 \pg 1 + \frac{n}{D^4}\int_{0}^{D}\frac{s^{3} \mathrm{d}s}{V(s)}   \pd \leq \frac{48/c^*}{n} \, ,
\end{equation}
which completes the proof of the case $q\geq 5$.
    Now assume that $1\leq q \leq 4$.
    Plugging the lower bounds from  \eqref{e:tesseratointon volume lower bound first} into Corollary \ref{cor: bounds for vertex transitive graphs with integrals and volumes}, we obtain
\begin{equation}
    \int_{0}^{D}\frac{s^{3} \mathrm{d}s}{V(s)} \leq 1 + \int_{1}^{R}\frac{s^{3} \mathrm{d}s}{c^*s^{q+1}} + \int_{R}^{D}\frac{s^{3} \mathrm{d}s}{c^*Rs^{q}} \, .
\end{equation}    
We obtain the result for $ q \in \ag 1,2,3,4 \ad$ doing a case by case analysis. 
\end{proof}

\medskip \textbf{Acknowledgements.} We thank Persi Diaconis, Roberto Imbuzeiro Oliveira, Perla Sousi and Matt Tointon for some useful discussions. This work started when the first two authors were at the University of Cambridge, during which time they were supported by EPSRC grant  EP/L018896/1. A first version of this paper was finished while the first author was in residence at the Mathematical Sciences Research Institute in Berkeley, California, during
the Spring 2022 semester on \emph{Analysis and Geometry of Random Spaces}, which was supported by the National Science Foundation under Grant No. DMS-1928930. N.B.  acknowledges the support from the Austrian Science Fund (FWF) grants 10.55776/F1002 on ``Discrete random structures: enumeration and scaling limits" and 10.55776/PAT1878824 on ``Random Conformal Fields'', and previously by FWF grant P33083, ``Scaling limits in random conformal geometry''.
J.H.'s research is supported by an NSERC grant. L. T.'s research was supported by FWF grant P33083, ``Scaling limits in random conformal geometry''.

\bibliography{bibliographyCTU}
\bibliographystyle{alpha}


\end{document}